\documentclass[12pt,leqno]{article}
\usepackage{amssymb,amsfonts,amsmath,amsthm,amscd,mathrsfs}
\usepackage{MnSymbol}
\usepackage{PSFigure,psfrag}
\usepackage{graphicx}
\usepackage{centernot}
\usepackage{stmaryrd}

\def\IR{{\mathbb R}}

\def\IT{{\mathbb T}}

\def\IZ{{\mathbb Z}}

\def\IX{{\mathbb X}}

\def\n{\noindent}

\def\dis{\displaystyle}

\def\fr{\mbox{\footnotesize $\dis\frac{1}{2}$}}

\def\ov{\overline}

\def\f{\footnotesize}

\def\r{\longrightarrow}

\def\wh{\widehat}
\def\wt{\widetilde}

\def\cB{{\cal B}}

\def\cD{{\cal D}}

\def\cJ{{\cal J}}

\dimendef\dimen=0

\newtheorem{theorem}{Theorem}[section]
\newtheorem{lemma}[theorem]{Lemma}

\newtheorem{proposition}[theorem]{Proposition}
\newtheorem{remark}[theorem]{Remark}

\begin{document}

\noindent
~

\bigskip
\begin{center}
{\bf ON LARGE DRIFT ENHANCEMENT FOR \\  THE RELAXATION TIME OF AVERAGES}
\end{center}

\begin{center}
Alain-Sol Sznitman$^1$ and Klaus Widmayer$^2$
\\[2ex]
Preliminary Draft
\end{center}

\begin{abstract}
In this note we prove in a broad analytic set-up a general result characterizing the asymptotic vanishing of the $L^2$-relaxation time of averages in the presence of a large incompressible drift-like term. The broad analytic set-up is analogous to that of Constantin-Kiselev-Ryzhik-Zlatos '08 or of Coti Zelati-Delgadino-Elgindi '20. A crucial role is played by the singular value gap method of Chatterjee '25, further developed by Eberle-Lörler '26.

\end{abstract} 

\vfill
\hfill October 2026

\vfill
\noindent
{\footnotesize
-------------------------------- \\
$^1\,$Departement of Mathematics, ETH Zurich, CH-8092 Zurich, Switzerland\\
$^2\,$Institute of Mathematics, University of Zurich, CH-8057 Zurich, Switzerland and Faculty of 

\vspace{-0.5ex}
\noindent
~~\!Mathematics, University of Vienna, A-1090 Vienna, Austria
}


\newpage
\thispagestyle{empty}
~

\newpage
\setcounter{page}{1}

\setcounter{section}{-1}
\section{Introduction}

Consider a time-homogeneous Markov process with a stationary probability $\mu$ on a state space $S$, inducing a strongly continuous contraction semi-group $(P_t)_{t \ge 0}$ on $L^2(\mu)$. As in \cite{EberLorl26}, one can quantify the speed of convergence to $\mu$ via the $L^2$-relaxation time:
\begin{equation}\label{0.1}
\tau_{\rm rel}\, \big((P_t)_{t \ge 0}\big) = \inf\{t \ge 0; \|P_t\|_{L^2_0(\mu) \r L^2_0(\mu)} \le 1/ e\}
\end{equation}
(where $L^2_0(\mu)$ stands for the subspace of $L^2(\mu)$ orthogonal to constants, which is left invariant by the semi-group due to the stationarity of $\mu$), and also by the $L^2$-relaxation time of averages:
\begin{equation}\label{0.2}
\begin{split}
\ov{\tau}_{\rm rel}\, \big((P_t)_{t \ge 0})  = & \inf\{\tau > 0; \, \| \ov{P}_t \|_{L^2_0(\mu) \r L^2_0(\mu)} \le 1 / e, \; \mbox{for all} \; t \ge \tau\}, \; \mbox{where}
\\
& \ov{P}_t \,f = \mbox{\f $\dis\frac{1}{t}$} \dis\int^t_0 P_s \,f \, ds, \; \mbox{for $f \in L^2(\mu)$ and $t > 0$}.
\end{split}
\end{equation}

\n
While $\|P_t\|_{L^2_0(\mu) \r L^2_0(\mu)}$ often decays exponentially with time, $\| \ov{P}_t\|_{L^2_0(\mu) \r L^2_0(\mu)}$ typically decays inversely proportional with time, see \cite{EberLorl26}, and the main purpose of this note is to characterize certain situations where ``turning on the drift boosts convergence to equilibrium measured in terms of (\ref{0.2})''.

\medskip
Similar questions concerning the influence of a ``large drift'' have a long history, see for instance \cite{BereHameNadi05}, \cite{HwanHwanShue05} and the references therein. In the spirit of the results discussed here, it was shown (in greater generality) in Theorem 0.3 of \cite{BereHameNadi05}, that when $U$ is a bounded smooth open subset of $\IR^d, d \ge 1, b(\cdot)$ a smooth $\IR^d$-valued function on $\IR^d$, with zero divergence, then, considering for $a \ge 0$, $\lambda_a$ the principal Dirichlet eigenvalue of $\Delta + a \, b \cdot \nabla$, one has the following statement:
\begin{equation}\label{0.3}
\begin{array}{l}
\mbox{$\lambda_a, a \ge 0$, remains bounded, when there is a non-zero $f$ in the Sobolev}
\\
\mbox{space $H^1_0(U)$ such that $b \cdot \nabla f = 0$, a.e.~on $U$. And otherwise,}
\\
\mbox{if no such $f$ exists, then $\lim\limits_{a \r \infty} \lambda_a = \infty$}.
\end{array}
\end{equation}

\n
On a question with a similar flavor, it was shown in Theorem 1.2 of \cite{ConsKiseRyzhZlat08} (in much greater generality) that when $\IT^d, d \ge 1$, stands for the $d$-dimensional torus, $b(\cdot)$ is a smooth $\IR^d$-valued function on $\IT^d$ with zero divergence, and $\tau_{\rm rel}(a)$ stands for the $L^2$-relaxation time of the diffusion on $\IT^d$ attached to $\Delta + a\, b \cdot \nabla$, one has the following behavior of the $L^2(\mu)$ relaxation time (with $\mu$ the uniform probability on $\IT^d$): 
\begin{equation}\label{0.4}
\begin{array}{l}
\mbox{$\tau_{\rm rel} (a), a \ge 0$, remains bounded away from $0$ when $b \cdot \nabla$ has a}
\\
\mbox{non-constant eigenfunction in the Sobolev space $H^1(\IT^d)$}, 
\\
\mbox{and otherwise $\lim\limits_{a \r \infty} \, \tau_{\rm rel}(a) = 0$}.
\end{array}
\end{equation}
In the present note, we show (in greater generality) that in the above example (with $\ov{\tau}_{\rm rel} (a)$ similarly defined as in (\ref{0.2})):
\begin{equation}\label{0.5}
\begin{array}{l}
\mbox{$\ov{\tau}_{\rm rel} (a), a \ge 0$, remains bounded away from $0$ when $b \cdot \nabla f = 0$ has a}
\\
\mbox{non-constant solution in $H^1(\IT^d)$, and otherwise $\lim\limits_{a \r \infty} \, \ov{\tau}_{\rm rel}(a) = 0$}.
\end{array}
\end{equation}
Thus, in particular, when identifying the influence of a large coefficient $a$ on $\ov{\tau}_{\rm rel} (a)$, unlike the case of $\tau_{\rm rel} (a)$, one only needs to consider the zero-eigenvalue, instead of all eigenvalues of $b \cdot \nabla$. Incidentally, the absence of a non-constant $f$ in $H^1(\IT^d)$ with $b \cdot \nabla f = 0$, is a weaker condition than the ergodicity of the flow generated by $b$, see Section 2, Example 2.

\medskip
In this note we prove in a broad analytic set-up, similar to that of \cite{ConsKiseRyzhZlat08} and \cite{CotiDelgElgi20} a general result implying (\ref{0.5}), see Theorem \ref{theo1.4}. A crucial role is played by the singular value gap methods put forward by \cite{Chat25}, see also \cite{Xu26}, and further developed by \cite{EberLorl26}. The main Theorem \ref{theo1.3} concerns the behavior of the square of the singular value gap in our abstract set-up. Together with the results of \cite{EberLorl26}, it yields Theorem \ref{theo1.4}. In Section 2, we discuss three examples. The last example is motivated by the work of the authors in \cite{SzniWidm26b}, concerning random stirring in the presence of a Galerkin two-dimensional Euler-type incompressible drift on a surface of constant energy and enstrophy. It leaves a challenging open question in Remark \ref{rem2.1}. The second example in Section 2 follows a suggestion of ChatGPT.

\section{The abstract result}
\setcounter{equation}{0}

In this section we introduce an abstract set-up similar to \cite{CotiDelgElgi20}  on p. 1211, or \cite{ConsKiseRyzhZlat08} on p.~647 and Section 2. Our main results here are Theorems {\ref{theo1.3} and \ref{theo1.4}. Theorem \ref{theo1.4} offers a broad generalization of (\ref{0.5}) and its proof relies on Theorem \ref{theo1.3} which concerns the square of singular value gaps. We now describe the set-up.

\medskip
We have a separable Hilbert space $H$ (playing the role of $L^2_0(\mu)$ in the Introduction), with scalar product denoted by $\langle \cdot, \cdot \rangle$ and norm by $\| \cdot \|$, as well as a self-adjoint, positive, unbounded operator with discrete spectrum denoted by $-D$ (for instance the negative Laplacian on functions of zero average on the torus $\IT^d$ in the example corresponding to (\ref{0.5})). Attached to $D$ is a non-decreasing sequence $\lambda_k, k \ge 1$, of positive eigenvalues tending to $\infty$, and a corresponding orthonormal basis of eigenfunctions $\varphi_k, k \ge 1$. In this fashion the spaces $H^1$ and $H^2$ that will be relevant to our analysis consist of vectors $f = \Sigma_{k \ge 1} \, c_k \, \varphi_k$ in $H$ such that $\| f \|_1^2 = \Sigma_k \,\lambda_k \, c^2_k < \infty$ or $\| f \|^2_2 = \Sigma_k \, \lambda^2_k \, c^2_k < \infty$ (in the context of (\ref{0.5}) the spaces correspond to the Sobolev spaces $\dot{H}^1(\IT^d)$ and $\dot{H}^2(\IT^d)$, where the superscript dots refer to functions of zero average on $\IT^d$). One then has
\begin{equation}\label{1.1}
\mbox{$\| f \|^2_2 = \| D f \|^2$ and $\| f \|^2_1 = - \langle Df, f \rangle$, for $ f \in H^2$}.
\end{equation}
Further, we have a possibly unbounded antisymmetric operator on $H$ denoted by $B$, that is, $i\, B$ is a self-adjoint operator on the complexification of $H$ (the operator $B$ corresponds to $b \cdot \nabla$ in the context of (\ref{0.5})). We assume that the domain of $B$ contains $H^1$ and that there exists $c_* > 0$ such that
\begin{equation}\label{1.2}
\| B f \| \le c_* \, \| f \|_1, \; \mbox{for all $f \in H^1$}.
\end{equation}
We further assume (this is condition (2.4) of  \cite{CotiDelgElgi20}) that there exists $c_{**} > 0$ such that
\begin{equation}\label{1.3}
| \langle D f, B f \rangle | \le c_{**} \, \| f \|^2_1, \; \mbox{for all $f \in H^2$}.
\end{equation}
The above conditions are for instance fulfilled in the example corresponding to (\ref{0.5}): this is a special case of \cite{CotiDelgElgi20}, Section 3.1, where the assumptions are shown to hold in the case of the Laplacian on a smooth compact manifold (playing the role of $D$) and a first order operator $B$ attached to a smooth divergence free drift on the manifold.

\medskip
In addition, see for instance \cite{ConsKiseRyzhZlat08} on p.~649, one considers for $a \ge 0$
\begin{equation}\label{1.4}
\begin{array}{l}
\mbox{$(Q^a_t)_{t \ge 0}$, the strongly continuous contraction semigroup on $H$ with}
\\
\mbox{generator $L_a = D + a B$}
\end{array}
\end{equation}
(in the context of (\ref{0.5}) it corresponds to the restriction of $(P_t)_{t \ge 0}$ to $L^2_0(\mu)$). 

\medskip
We introduce for $a \ge 0$ the key quantity
\begin{equation}\label{1.5}
\gamma_a = \inf \{ \| L_a f\|^2; \; f \in H^2, \| f \| = 1\}
\end{equation}
(this is the square of the singular value gap mentioned in the Introduction).

\medskip
The behavior of $\gamma_a$ as $a \r \infty$ will be at the heart of Theorem \ref{theo1.3} below. We start with the (somewhat classical) fact that the infimum in (\ref{1.5}) actually is a minimum.

\begin{lemma}\label{lem1.1}
For any $a \ge 0$, there is an $f_a$ in $H^2$ such that
\begin{equation}\label{1.6}
\mbox{$\gamma_a = \|L_a \,f_a\|^2$, and $\|f_a\| = 1$}.
\end{equation}
\end{lemma}

\begin{proof}
We first note that for any $a \ge 0$, there is $c_a > 0$ such that
\begin{equation}\label{1.7}
a \, \|B f\| \le \mbox{\f $\dis\frac{1}{2}$} \; \| D f \| + c_a \,\| f \|, \; \mbox{for all $f \in H^2$}.
\end{equation}
Indeed, by (\ref{1.2}) we have $\| Bf \| \le c_* \| f \|_1$, and with the notation above (\ref{1.1}), one has for all $a \ge 0$ a non-negative $\wt{c}_a$ such that for all sequences of reals $c_k, k \ge 1$, $(c_* a)^2\, \Sigma_k \, \lambda_k \, c^2_k \le \frac{1}{4} \, \Sigma_k \, \lambda_k^2 \, c_k^2 + \wt{c}_a \, \Sigma_k \, c_k^2$, and (\ref{1.7}) follows.

\medskip
As a result of (\ref{1.7}), we see that for fixed $a \ge 0$, for all $f \in H^2$ with $\| f \| = 1$
\begin{equation}\label{1.8}
\|L_a f\| = \| D f + a \, B f\| \ge \| Df \| - a \| B f \| \stackrel{(\ref{1.7})}{\ge} \fr \; \| D f \| - c_a.
\end{equation}
Thus, for (\ref{1.5}), when $g_n$ is a minimizing sequence in $H^2$ such that $\| g_n \| = 1$ and $\| L_a \, g_n\|^2 \r \gamma_a$, we see by (\ref{1.8}) that $\| D g_n\|$ is bounded. We can then extract a subsequence $g_{n_k}$ converging weakly in $H^2$ to $g_*$, see \cite{Lax02} on p.~104, and strongly in $H^1$ and $H$ to $g_*$, so that $\| D g_*\|^2 = \|  g_*\|^2_2 \le \liminf_k \, \|g_{n_k}\|^2_2$. We then have 
\begin{equation}\label{1.9}
\gamma_a = \lim\limits_k \|L_a \, g_{n_k} \|^2 = \lim\limits_k (\|g_{n_k}\|^2_2 + 2a \langle D g_{n_k}, Bg_{n_k}\rangle + a^2 \|Bg_{n_k}\|^2).
\end{equation}
Now, $Dg_{n_k}$ converges weakly in $H$ to $Dg_*$ and by (\ref{1.2}) $Bg_{n_k}$ converges strongly to $Bg_*$ in $H$, so that $\langle D g_{n_k},Bg_{n_k}\rangle$ tends to $\langle Dg_*, Bg_*\rangle$. Together with the observation above (\ref{1.9}), we see that $\| g_*\| = 1$ and
\begin{equation}\label{1.10}
\gamma_a \ge \| Dg_*\|^2 + 2a \langle Dg_*, Bg_*\rangle + a^2 \|Bg_*\|^2 = \|L_a g_*\|^2.
\end{equation}
The claim (\ref{1.6}) follows with $f_a = g_*$, and Lemma \ref{lem1.1} is proved. 
\end{proof}

The next proposition will be helpful in the proof of Theorem \ref{theo1.3} below, when showing that if $\gamma_a$ does not tend to infinity as $a$ goes to infinity, then there is a non-trivial solution $f_*$ in $H^1$ of $Bf_* = 0$.

\begin{proposition}\label{prop1.2}
For $a \ge 0$ and $f_a$ as in \rm{(\ref{1.6})}, one has
\begin{equation}\label{1.11}
\gamma_a \ge (\|f_a \|_2 - a \,\|Bf_a \|)^2 + a \mbox{\f $\dis\frac{\|Bf_a\|}{\|f_a\|_2}$} \; \| f_a \|_1^4,
\end{equation}
and, see (\ref{1.3}) for notation,
\begin{equation}\label{1.12}
\gamma_a  + 2 a \, c_{**} \, \|f_a\|_1^2 \ge a^2 \|Bf_a\|^2 + \|f_a \|^2_2.
\end{equation}
\end{proposition}

\begin{proof}
We first prove (\ref{1.11}). We note that $B$ is antisymmetric, its domain contains $H^1$, and $f_a$ belongs to $H^2$, so that $\langle f_a, Bf_a\rangle = 0$. Hence, for all $a \ge 0$ and $\lambda$ in $\IR$, one has
\begin{equation}\label{1.13}
\begin{array}{lcl}
\gamma_a &\!\!\!\! \stackrel{(\ref{1.6})}{=} &\!\!\!\! \| Df_a + a \,Bf_a\|^2 = \|D f_a \|^2 + 2 a \langle (D+ \lambda)\,f_a, Bf_a \rangle + a^2 \,\|Bf_a\|^2
\\[1ex]
&\!\!\!\! \ge &\!\!\!\!  \|D f_a \|^2 - 2 a \,\|(D+ \lambda) \,f_a \| \; \|Bf_a \| + a^2 \,\|Bf_a\|^2
\\[1ex]
&\!\!\!\! = &\!\!\!\!  (\| Df_a \| - a \, \|Bf_a \|)^2 + 2a \, \|Bf_a \| \; (\|Df_a \| - \|(D + \lambda) \,f_a\|).
\end{array}
\end{equation}
Choosing $\lambda = -\langle Df_a, f_a \rangle \stackrel{(\ref{1.1})}{=} \|f_a \|^2_1$ and using $\|f_a\| =1$, we find that
\begin{equation}\label{1.14}
\begin{split}
\|Df_a + \lambda f_a\|^2 & = \|Df_a\|^2 + 2\lambda \, \langle Df_a, f_a \rangle + \lambda^2
\\[1ex]
& = \|f_a \|^2_2 - \|f_a \|_1^4 .
\end{split}
\end{equation}
We can now come back to the last line of (\ref{1.13}), and with $\|Df_a\| = \|f_a\|_2$ find:
\begin{equation}\label{1.15}
\begin{split}
\gamma_a & \ge \|f_a\|_2 - a \, \|Bf_a\|)^2 + 2 a \, \| Bf_a \| \; \|f_a\|^4_1 \, / \, \{\|f_a\|_2 + (\|f_a\|^2_2 - \|f_a\|_1^4)^{1/2}\}
\\[0.5ex]
& \ge (\|f_a\|_2 - a\,\|Bf_a\|)^2 + a \; \mbox{\f $\dis\frac{\| Bf_a\|}{\|f_a\|_2}$} \; \|f_a\|^4_1 ,
\end{split}
\end{equation}
and this proves (\ref{1.11}).  As for (\ref{1.12}), we note using (\ref{1.3}) in the second line that
\begin{equation}\label{1.16}
\begin{split}
\gamma_a & \stackrel{(\ref{1.6})}{=} \| Df_a + a\, Bf_a\|^2 = \|f_a\|^2_2 + a^2 \, \|Bf_a\|^2 + 2a \langle Df_a, Bf_a \rangle
\\[0.5ex]
& \stackrel{(\ref{1.3})}{\ge} \|f_a\|^2_2 + a^2 \,\|Bf_a\|^2 - 2a \, c_{**} \|f_a \|^2_1,
\end{split}
\end{equation}
and the claim (\ref{1.12}) follows. This completes the proof of Proposition \ref{prop1.2}.
\end{proof}

We now come to the main
\begin{theorem}\label{theo1.3}
The following two conditions are equivalent
\begin{align}
& \lim\limits_{a \r \infty} \; \gamma_a = \infty, \label{1.17}
\\[1ex]
& \mbox{the only $f$ in $H^1$ such that $Bf = 0$ is $f = 0$}. \label{1.18}
\end{align}
Moreover, if {\rm (\ref{1.17})} (or equivalently {\rm  (\ref{1.18})}) does not hold, then
\begin{equation}\label{1.19}
\sup\limits_{a \ge 0} \; \gamma_a < \infty.
\end{equation}
\end{theorem}

\begin{proof}
We first show that
\begin{equation}\label{1.20}
\mbox{if $\liminf\limits_{a \r \infty} \; \gamma_a < \infty$, then there exists $f_*$ in $H^1 \backslash \{0\}$ such that $Bf_* = 0$}.
\end{equation}
To this end, assuming that $\liminf_{a \r \infty} \gamma_a < \infty$, we consider $a_n \r \infty$ such that $\lim_n \gamma_{a_n} = \gamma_* < \infty$. As a first step we note that necessarily
\begin{equation}\label{1.21}
\|f_{a_n}\|_2 = o(a_n), \; \mbox{as $n \r \infty$}.
\end{equation}
Indeed, otherwise, by (\ref{1.11}), extracting a subsequence still denoted by $a_n$, with at least linear growth in $a_n$, we would have $\|f_{a_n}\|_2 \sim a_n \|Bf_{a_n}\|$, and by (\ref{1.11}), $\|f_{a_n}\|_1$ would be uniformly bounded. However, this would contradict (\ref{1.12}). Hence, (\ref{1.21}) holds.

\medskip
Since (\ref{1.21}) holds, there are two possibilities:

\medskip
Either both $\|f_{a_n}\|_2$ and $a_n \|Bf_{a_n}\|$ are uniformly bounded.

\medskip
In this case, see \cite{Lax02} on p.~104, we can extract a subsequence, still denoted by $a_n$, such that $f_{a_n}$ converges weakly in $H^2$, and hence strongly in $H^1$ to $f_*$, and hence $\|Bf_*\| = \lim_n \|Bf_{a_n}\| = 0$. So $\|f_*\| = 1$ and the conclusion of (\ref{1.20}) holds.

\medskip
Or, along a subsequence, still denoted by $a_n$, both $\|f_{a_n}\|_2$ and $a_n \|Bf_{a_n}\|$ tend to infinity and are $o(a_n)$. In this case, $\|Bf_{a_n}\|$ tends to $0$, and by (\ref{1.11}) the sequence $\|f_{a_n}\|_1$ remains bounded. We can then find, by the same reference as above, a further subsequence, still denoted by $a_n$, and $f_*$ in $H^1$ with $\|f_*\| = 1$, which is a weak limit on $H^1$ and a strong limit in $H$ of $f_{a_n}$. Then, for any $\varphi$ finite linear combination of the eigenfunctions $\varphi_k, k \ge 1$, see above (\ref{1.1}), we find that
\begin{equation}\label{1.22}
\begin{array}{l}
\langle Bf_{a_n}, \varphi\rangle \underset{n}{\longrightarrow} 0, \; \mbox{and}
\\
\langle Bf_{a_n}, \varphi \rangle = - \langle f_{a_n}, B \varphi \rangle \underset{n}{\longrightarrow} - \langle f_*, B \varphi\rangle = \langle Bf_*, \varphi \rangle.
\end{array}
\end{equation}
This shows that $\langle Bf_*, \varphi \rangle = 0$, for all $\varphi$ as above, so that $f_*$ satisfies the conclusion of (\ref{1.20}). This completes the proof of (\ref{1.20}), and shows by contradiction that (\ref{1.18}) implies (\ref{1.17}).

\medskip
We will now prove that
\begin{equation}\label{1.23}
\mbox{if there is $f_* \not= 0$ in $H^1$ with $Bf_* = 0$, then $\sup\limits_{a \ge 0} \gamma_a < \infty$}.
\end{equation}
This will complete the proof of the equivalence of (\ref{1.17}) and (\ref{1.18}), and the fact if (\ref{1.18}) (or (\ref{1.17})) does not hold then (\ref{1.19}) holds. Theorem \ref{theo1.3} will then be proved.

\medskip
We thus show (\ref{1.23}). Without loss of generality, we can assume that $\|f_*\| = 1$. Set $f_a(t) = Q^a_t \, f_*$, for $t \ge 0$ (see (\ref{1.4}) for notation). By \cite{{ConsKiseRyzhZlat08}} on p.~651, we have for $t \ge 0$
\begin{equation}\label{1.24}
\langle f_a(t), f_*\rangle = 1 + \dis\int^t_0 \langle D f_a(s), f_*\rangle ds \ge 1 - \fr \;\dis\int^t_0 (\|f_a(s)\|_1^2 + \|f_*\|^2_1) \,ds,
\end{equation}
as well as
\begin{equation}\label{1.25}
\dis\int^\infty_0 \|f_a(s) \|^2_1 \, ds \le \fr \,.
\end{equation}
Thus, choosing $t \le \tau_* = (2 \|f_*\|^2_1)^{-1}$, we see that
\begin{equation}\label{1.26}
\langle f_a(t), f_* \rangle \ge \fr, \; \mbox{for $0 \le t \le \tau_*$}.
\end{equation}
Since $f_* \in H^1$, $f_a(t)$ belongs to $L^2([0,T], H^2) \cap C([0,T], H^1)$, for any $T > 0$, by Theorem 2.1 of \cite{ConsKiseRyzhZlat08}, and we have, see (15) on p.~421 of \cite{Lax02} (or also (35) of \cite{EberLorl26})
\begin{equation}\label{1.27}
L_a \,\Big(\dis\int^t_0 f_a(s) \,ds\Big) = f_a(t) - f_a(0) = f_a(t) - f_*, \; \mbox{for $t \ge 0$}.
\end{equation}
Hence, by (\ref{1.5}), we have
\begin{equation}\label{1.28}
\gamma_a^{1/2} \,\Big\| \dis\int^t_0 f_a(s)\,ds\Big\| \le \| f_a(t) - f_*\| \le 2, \; \mbox{for $t > 0$}.
\end{equation}
On the other hand, by (\ref{1.26}), $\langle \int^t_0 f_a (s)\,ds, f_*\rangle \ge t/2$, for $0 \le t \le \tau_*$, so that
\begin{equation}\label{1.29}
\Big\|\dis\int^t_0 f_a(s)\, ds\Big\| \ge t/2, \; \mbox{for $0 \le t \le \tau_*$}.
\end{equation}
Inserting this inequality into (\ref{1.28}), and recalling the definition of $\tau_*$ above (\ref{1.26}), we find that for all $a \ge 0$,
\begin{equation}\label{1.30}
\gamma_a^{1/2} \le 4/\tau_* = 8 \, \|f_*\|^2_1 .
\end{equation}
This proves (\ref{1.23}), and as already explained, completes the proof of Theorem \ref{theo1.3}.
\end{proof}

One can now introduce the relaxation time of averages for $(Q^a_t)_{r \ge 0}$ via
\begin{equation}\label{1.31}
\ov{\tau}_{\rm rel}\, \big((Q_t^a)_{t \ge 0}\big) = \inf\{\tau \ge 0; \; \Big\| \mbox{\f $\dis\frac{1}{t}$} \dis\int^t_0 Q^a_s \, ds\Big\|_{H \r H} \le 1/e, \; \mbox{for all $t \ge \tau\}$},
\end{equation}
where $\| \; \|_{H \r H}$ refers to the operator norm on $H$.

\medskip
If one defines the lowest singular value of $L_a = D + a B$, see (\ref{1.4}), via
\begin{equation}\label{1.32}
\mbox{$s(L_a) = \sqrt{\gamma_a}$, see (\ref{1.5}) for notation},
\end{equation}
the proof of Theorem 7 of \cite{EberLorl26} is readily adapted and shows that for all $a \ge 0$
\begin{equation}\label{1.33}
2(1 - e^{-1}) / s(L_a) \le \ov{\tau}_{\rm rel} \,\big((Q^a_t)_{t \ge 0}\big) \le 2e / s(L_a),
\end{equation}
and one has the inequality
\begin{equation}\label{1.34}
\Big\| \mbox{\f $\dis\frac{1}{t}$}  \dis\int^t_0 Q^a_s f\, ds\Big\| \le \mbox{\f $\dis\frac{2}{ts(L_a)}$} \;\|f\|, \; \mbox{for $f \in H$ and $t > 0$}.
\end{equation}
This sets the stage for the main application of Theorem \ref{theo1.3} to the relaxation time of averages. We refer to (\ref{1.4}) and (\ref{1.31}) for notation.

\begin{theorem}\label{theo1.4}
The following two conditions are equivalent:
\begin{align}
& \lim\limits_{a \r \infty} \ov{\tau}_{\rm rel}\, \big((Q^a_t)_{t \ge 0}\big) = 0, \label{1.35}
\\[1ex]
& \mbox{the only $f$ in $H^1$ such that $Bf = 0$ is $f = 0$}. \label{1.36}
\end{align}
Moreover, if {\rm (\ref{1.35})} (or equivalently {\rm (\ref{1.36})}) does not hold, then
\begin{equation}\label{1.37}
\inf\limits_{a \ge 0} \ov{\tau}_{\rm rel} \,\big((Q^a_t)_{t \ge 0}\big) > 0.
\end{equation}
\end{theorem}

\begin{proof}
This is the application of Theorem \ref{theo1.3} and (\ref{1.4}). 
\end{proof}

\section{Some examples}
\setcounter{equation}{0}

In this section we discuss some examples. The first two examples bring into play the $d$-dimensional torus $\IT^d = (\IR / 2 \pi \IZ)^d, d \ge 1$, and the Laplacian in the presence of a smooth divergence free drift. The assumptions of the beginning of Section 1, including (\ref{1.1}) - (\ref{1.3}) (with the natural Sobolev spaces) are actually known to hold in the much broader set-up of the Laplacian on a smooth compact Riemannian manifold in the presence of a smooth divergence free drift, see Section 3.1 of \cite{CotiDelgElgi20}, and also \cite{ConsKiseRyzhZlat08}. The third example pertains to advection and random stirring on the smooth compact connected manifold of constant enstrophy and energy, see also \cite{SzniWidm26b}.

\bigskip\n
{\bf Example 1:} \\
We consider the two-dimensional torus $\IT^2$ and $\alpha > 0$ an irrational number. We let $H = L^2_0(\IT^2)$ stand for the space of square integrable functions on $\IT^2$ with zero-average, and $H^1 = \dot{H}^1(\IT^2)$, $H^2 = \dot{H}^2(\IT^2)$ stand for the corresponding Sobolev spaces (the superscript $\cdot$ refers to zero-average functions). We let $D$ correspond to the Laplacian $\Delta$ and $B$ to $ \partial_x + \alpha \partial_y$ (acting on zero-average functions). Given $a \ge 0$, we consider as in the Introduction, see (\ref{0.4}), (\ref{0.5}), $\tau_{\rm rel}(a)$, the $L^2$-relaxation time of $\Delta + a(\partial_x + \alpha \partial_y)$ and $\ov{\tau}_{\rm rel}(a)$ the corresponding $L^2$-relaxation time of averages. As we now explain, in this example
\begin{equation}\label{2.1}
\lim\limits_{a \r \infty} \ov{\tau}_{\rm rel}(a) = 0, \; \mbox{but} \; \inf\limits_{a \ge 0} \tau_{\rm rel}(a) > 0.
\end{equation}
Thus, enhancement by a large $a$ only takes place for time-averages.

\medskip
To prove the first statement of (\ref{2.1}), we note that since $\alpha$ is irrational, the only solution of $Bf = 0$ in $H^2$ is $f = 0$ (such an $f$ can be approximated through convolution with smooth solutions of the same equation, which are constant due to the ergodicity of the flow generated by $\partial_x + \alpha \partial_y$). On the other hand, for $k \not= 0$ in $\IZ^2$, $e^{ik \cdot x}$ is an eigenfunction of $iB$ attached to the eigenvalue $-(k_1 + \alpha k_2)$ $(\not= 0)$, if $k = (k_1,k_2)$, and it belongs to the complexification of $H^1$. The second statement now follows from \cite{ConsKiseRyzhZlat08} on p.~651.

\bigskip\n
{\bf Example 2:}\\
We consider the three-dimensional torus $\IT^3$ and display an example of a smooth incompressible vector field $b(\cdot)$ on $\IT^3$ such that
\begin{equation}\label{2.2}
\begin{array}{l}
\mbox{there is a continuous non-constant function $F$ on $\IT^3$, which remains constant}\\
\mbox{along any trajectory generated by $b(\cdot)$}.
\end{array}
\end{equation}
In particular, the flow generated by $b(\cdot)$, which preserves the Lebesgue measure on $\IT^3$, is not ergodic. But,
\begin{equation}\label{2.3}
\mbox{if $G \in H^1(\IT^3)$ and $b \cdot \nabla G = 0$, then $G$ is constant a.e.,}
\end{equation}
and, in view of Theorem \ref{theo1.4}, enhancement by a large $a$ takes place for time-averages of the $L^2$-semigroup generated by $\Delta + a \, b \cdot \nabla$. As an aside, we refer to Section 6 of \cite{ConsKiseRyzhZlat08} for various examples of ergodic relaxation enhancing flows.

\medskip
To construct the above vector field $b(\cdot)$, we consider $\alpha > 0$ a Liouville irrational number, see for instance \cite{HardWrig60}, at the bottom of p.~161, as well as two sequences of positive integers $p_n,q_n, n \ge 1$, such that setting $k_n = (p_n,q_n)$, one has
\begin{align}
& \mbox{$| k_n | \ge n^2$, for $n \ge 1$, and} \label{2.4}
\\[1ex]
&\mbox{for all $M \ge 1$, $p_n - \alpha \, q_n = o(|k_n|^{-M}$), as $n \r \infty$}. \label{2.5}
\end{align}
One then defines
\begin{align}
u(y) & = \Sigma_n \,n^{-2} \cos( k_n \cdot y), \; \mbox{for $y \in \IT^2$, and} \label{2.6}
\\[1ex]
v(y) & = - \,\Sigma_n \,n^{-2} (p_n - \alpha \, q_n) \sin ( k_n \cdot y), \; \mbox{for $y \in \IT^2$}. \label{2.7}
\end{align}
Due to (\ref{2.4}), (\ref{2.5}), one thus finds that
\begin{equation}\label{2.8}
\mbox{$u$ is continuous on $\IT^2$ but not in $H^1(\IT^2)$},
\end{equation}
and setting
\begin{align}
& y _0 = (1,-\alpha) \in \IR^2, \label{2.9}
\\[1ex]
&\mbox{$v$ is smooth on $\IT^2$ and coincides with the derivative of $u$ in the direction $y_0$}. \label{2.10}
\end{align}
We then define the smooth vector field on $\IT^3 = \IT^1 \times \IT^2$:
\begin{equation}\label{2.11}
b(x,y) = \big(v(y) , y_0\big), \; \mbox{for $x \in \IT^1, y \in \IT^2$}.
\end{equation}
Hence, $b$ is incompressible:
\begin{equation}\label{2.12}
{\rm div} \, b = 0,
\end{equation}
and as one readily checks by approximation and (\ref{2.9}), (\ref{2.10})
\begin{equation}\label{2.13}
\mbox{the flow generated by $b$ preserves the function $F(x,y) = e^{ i(x - u(y))}$}
\end{equation}
(formally $b \cdot \nabla F (x,y) = [i (v - y_0 \cdot \nabla u) F] (x,y) \stackrel{(\ref{2.10})}{=} (v(y) - v(y))F(x,y) = 0$). This proves the claim (\ref{2.2}).

\medskip
Now consider $G$ as in (\ref{2.3}). We can expand $G$ in Fourier series in the first variable $x$, and write for a sequence $g_m$ in $H^1(\IT^2)$:
\begin{align}
& \mbox{$G(x,y) = \Sigma_m \, g_m(y) \, e^{ i m x}, x \in \IT^1, y \in \IT^2$, so that} \label{2.14}
\\[1ex]
& 0 = b \cdot \nabla G(x,y) = \Sigma_m \big(v(y) \,  i  m \, g_m(y) + y_0 \cdot \nabla g_m(y)\big) \,e^{ i  m x} \; \mbox{in} \; L^2 (\IT^3). \label{2.15}
\end{align}
As a result, each function of $y$ under the parenthesis in (\ref{2.15}) equals $0$, a.e.. Multiplying by $\exp\{ i m \,u(y)\}$, we find that
\begin{equation}\label{2.16}
e^{i m \,u(y)} \big(v(y) \, i m \,g_m(y) + y_0 \cdot \nabla g_m(y)\big) = 0, \; \mbox{for a.e. $y \in \IT^2$ and $m \in \IZ$}.
\end{equation}
Using (\ref{2.10}), the above expression coincides with the derivative in the $y_0$ direction, in the $L^2(dy)$ sense, of $h_m(y) = e^{i m \,u(y)} g_m(y)$. Since $y_0 = (1,-\alpha)$, with $\alpha$ irrational, the flow generated by $y_0$ on $L^2(dy)$ is ergodic, and for each $m \in \IZ$ there is a constant $c_m$ such that $h_m = c_m$ a.e.. When $m \not= 0$, since $g_m \in H^1(\IT^2)$ but $u \notin H^1(\IT^2)$ by (\ref{2.8}), we find that $c_m = 0$. And when $m = 0$, we conclude $g_0 = c_0$, a.e.. Thus, coming back to (\ref{2.14}), this concludes the proof of (\ref{2.3}).

\bigskip\n
{\bf Example 3:}\\
The third and last example pertains to advection and random stirring as in the work of the authors in \cite{SzniWidm26b}, see for instance Appendix B of \cite{SzniWidm26b}. We devise in (\ref{2.26}) below, a necessary condition for the enhancement through advection of the $L^2$-relaxation time of averages, for a certain random stirring on a space corresponding to ``fixed enstrophy and energy'', see (\ref{2.18}), and (\ref{2.23}) below.

\medskip
We have an even integer $N = 2n$, where $n \ge 4$, as well as positive $\lambda_\ell$, $1\le \ell \le N$, such that $1 = \lambda_1 = \lambda_2 < \lambda_3 = \lambda_4 < \dots < \lambda_{N-1} = \lambda_N$. We have the scalar products on $\IR^N$, denoted by $\langle \cdot, \cdot \rangle$ and $\langle \cdot, \cdot \rangle_{-1}$ associated to the square norms
\begin{equation}\label{2.17}
|x|^2 = \Sigma_\ell \, x_\ell^2 \ge |x|^2_{-1} = \Sigma_\ell \,x_\ell^2 / \lambda_\ell, \; \mbox{for $x \in \IR^N$}.
\end{equation}
We denote by $\mu$ the Gaussian measure on $\IR^N$, under which the coordinates are i.i.d. centered variables with variance $\frac{1}{2}$. For $0 < v < u$ such that $u/v \not= \lambda_\ell, 1 \le \ell \le N$, the space
\begin{equation}\label{2.18}
\IX_{u,v} = \{x \in \IR^N; |x|^2 = u, |x|^2_1 = v\}
\end{equation}
is a smooth compact connected, submanifold of $\IR^N$, see (A.10) of \cite{SzniWidm26b}. Further, one can define the regular conditional probability $\mu_{u,v}$ of $\mu$ given $|x|^2 = u$, $|x|^2_{-1} = v$, which is supported by $\IX_{u,v}$, see (A.12) of \cite{SzniWidm26b}.

\medskip
Denoting by $\cJ$ the collection of triples $J = (k,\ell,m)$ in $\{1, \dots, N\}$ with $\lambda_k < \lambda_\ell < \lambda_m$, one can define for $J$ in $\cJ$ the vector field (there are three terms in the sum):
\begin{equation}\label{2.19}
T_J (x) = \dis\sum\limits_{(\mbox{\f $a,b,c)$ cyclic permutation of $J$}} x_a \, x_b \,\Big(\mbox{\f $\dis\frac{1}{\lambda_a} - \dis\frac{1}{\lambda_b}\Big)$}\,\partial_c\,,
\end{equation}
as well as for $1 \le i \le n$
\begin{equation}\label{2.20}
R_i (x) = x_{2i} \, \partial_{2 i - 1} - x_{2i -1} \, \partial_{2i}.
\end{equation}
Any vector field $Z$ in the collection $T_J, J \in \cJ$, $R_i, 1 \le i \le n$, satisfies for all $x$
\begin{equation}\label{2.21}
{\rm div} \, Z = 0, \langle x, Z(x) \rangle = 0, \langle x, Z(x)\rangle_{-1} = 0.
\end{equation}
In addition, see Lemma B.1 of \cite{SzniWidm26b}, for $f,g$ smooth functions on $\IX_{u,v}$, one has
\begin{equation}\label{2.22}
\dis \int \!Zf \; g\; d\mu_{u,v} = - \dis\int \! f \; Zg \; d\mu_{u,v}.
\end{equation}
One then defines the diffusion operator $\cD$ on $\IX_{u,v}$ via
\begin{equation}\label{2.23}
\cD = \fr \;(\Sigma_{J \in \cJ} \,T^2_J + \Sigma^n_1 \, R^2_i).
\end{equation}
It is uniformly elliptic and self-adjoint in $L^2(\mu_{u,v})$, see (B.9), (B.10) of \cite{SzniWidm26b}.

\medskip
Motivated by the case of the Galerkin projection for the non-linear term of the two-dimensional incompressible Navier-Stokes equation in vorticity form, on  a possibly thin torus, as explained in Remark 1.1. of \cite{SzniWidm26b}, one considers real numbers $s_J, J \in \cJ$, not all zero, and the drift operator
\begin{equation}\label{2.24}
\cB = \Sigma_{J \in \cJ}\, s_J \,T_J \quad  \mbox{(it fulfills (\ref{2.21}), (\ref{2.22}))}.
\end{equation}
One is now interested in the behavior for large $a \ge 0$ of
\begin{equation}\label{2.25}
\begin{split}
\ov{\tau}_{\rm rel}\,(a) = &\;  \mbox{$L^2(\mu_{u,v})$-relaxation time of averages for the diffusion on $\IX_{u,v}$} \\
&\;  \mbox{attached to $a \cB + \cD$}.
\end{split}
\end{equation}
The abstract set-up of the beginning of Section 1 holds, once one notes that for any $Z, Z^\prime$ in the collection above (\ref{2.21}), and smooth $f$ on $\IX_{u,v}$, one has by (\ref{2.22}) $\int Z^2 f \; Z^\prime \! f \; d\mu_{u,v} = - \int Z \! f \; [Z, Z^\prime]  f \, d\mu_{u,v}$, and (\ref{1.3}) is fulfilled.

\medskip
We will now state a necessary condition on the $s_J$ in (\ref{2.24}) for a vanishing limit of $\ov{\tau}_{\rm rel}(a)$ as $a \r \infty$. For this purpose we need some terminology. Given $J = (k,\ell,m)$ in $\cJ$ we denote by $\wh{J} = \{k,\ell,m\} \subseteq \{1, \dots, N\}$ its support, and we say that $J \not= J^\prime$ in $\cJ$ are neighbors if $\wh{J} \cap \wh{J}\,{\!^\prime} \not= \emptyset$.This provides a graph structure on $\cJ$, and we look at the subgraph corresponding to the vertex set $\cJ^s = \{ J \in \cJ; s_J \not= 0\}$ and the previous notion of neighbors. Given a connected component of $\cJ^s$, we refer to its support as the union of the supports $\wh{J}$ for $J$ in the connected component. We can now state:
\begin{equation}\label{2.26}
\mbox{if $\lim\limits_{a \r \infty} \ov{\tau}_{\rm rel} \,(a) = 0$, then $\cJ^s$ is connected and its support is $\{1, \dots, N\}$}.
\end{equation}
Indeed, if the conclusion fails, denote by $S$ the support of one connected component of $\cJ^s$ and by $c$ an element of $\{1, \dots, N\} \backslash S$. Consider the function 
\begin{equation}\label{2.27}
F(x) = \Sigma_{a \in S} \, x^2_a, \; \mbox{for $x \in \IX_{u,v}$}.
\end{equation}
Since $S$ is the support of a connected component of $\cJ^s$, then for $J \in \cJ$, $s_J \not= 0$ implies that either $\wh{J} \subseteq S$ or $\wh{J} \cap S = \emptyset$, so that $T_J \, F = 0$. We thus find that 
\begin{equation}\label{2.28}
\mbox{$\cB F = 0$, on $\IX_{u,v}$}. 
\end{equation}
As we now explain, $F$ is not constant on $\IX_{u,v}$. To this end, we choose an $x^0$ in $\IX_{u,v}$ with all coordinates different from $0$ (see Appendix A of \cite{SzniWidm26b}), and note that in $x^0$, the gradients of the functions $|x|^2, |x|^2_{-1}, F(x)$ are respectively equal to $2x^0, 2 \Lambda^{-1} x^0, 2 \wt{x}^{\, 0}$, where  $\Lambda^{-1}$ is the diagonal matrix with diagonal entries $\lambda^{-1}_\ell, 1 \le \ell \le N$, and $\wt{x}^{\, 0}_\ell = x^0_\ell \, 1\{\ell \in S\}, 1 \le \ell \le N$. These three vectors are linearly independent: Indeed, taking $a,b$ in $\wh{J} \subseteq S$ (so that $\lambda_a \not= \lambda_b$ by the fact that $J \in \cJ$), and $c$ outside $S$, we note that
\begin{equation*}
\left| \begin{array}{llc}
x^0_a & x^0_a / \lambda_a & x^0_a
\\[1ex]
x^0_b & x^0_b / \lambda_b & x^0_b
\\[1ex]
x^0_c & x^0_c / \lambda_c & 0
\end{array}\right| \; \not= 0\,.
\end{equation*}
So the gradient of $F$ at $x^0$ does not belong to the space generated by $x^0, \Lambda^{-1} x^0$, and $F$ is a smooth non-constant function on $\IX_{u,v}$ for which (\ref{2.28}) holds. By Theorem \ref{theo1.4} this implies that $\ov{\tau}_{\rm rel}\,(a), a \ge 0$, is uniformly bounded away from zero: a contradiction to the assumption in (\ref{2.26}). This completes the proof of (\ref{2.26}).

\begin{remark}\label{rem2.1} \rm So, (\ref{2.26}) provides a simple necessary condition for an enhancement through large $a$ of the $L^2$-relaxation time of averages of the diffusion on $\IX_{u,v}$ attached to $a \, \cB + \cD$. Can one find an example where such an enhancement does take place? \hfill $\square$
\end{remark}

\medskip\n
{\bf Use of generative artificial intelligence.} \\
Our only use of AI tools pertains to Example 2, which was suggested to us in its essence by ChatGPT. We have adapted it for the presentation here: all text is written by the authors.

\end{document}